\documentclass[preprint,12pt]{elsarticle}

\usepackage{epsfig}

\usepackage{amssymb}
\usepackage{amsmath}
\usepackage{amsthm}
\usepackage{tabularx}
\usepackage{float}
\usepackage{ragged2e}
\usepackage{xcolor}
\usepackage{url}

\theoremstyle{definition}
\newtheorem{defi}{Definition}[section]

\theoremstyle{remark}

\theoremstyle{plain}
\newtheorem*{fact}{Fact$^*$}
\newtheorem{lem}[defi]{Lemma}
\newtheorem{thm}[defi]{Theorem}

\newtheorem{rem}[defi]{Remark}

\begin{document}

\begin{frontmatter}



\title{\Large Chains without regularity \tnoteref{l1}} 

\author{A. Della Corte\fnref{adc}}
\ead{alessandro.dellacorte@unicam.it}
\fntext[adc]{Mathematics Division, School of Sciences and Technology, University of Camerino (Italy).
Address: via Madonna delle Carceri 9A, Camerino (MC), Italy. ORCID: 0000-0002-1782-0270 (corresponding author).}

\author{M. Farotti\fnref{mf}}
\ead{marco.farotti@studenti.unicam.it.}
\fntext[mf]{Doctoral School in Computer Sciences and Mathematics, University of Camerino (Italy).
Address: via Madonna delle Carceri 9A, Camerino (MC), Italy. ORCID: 0009-0001-5000-2827.}

\begin{abstract}
We study chain-recurrence and chain-transitivity in compact dynamical systems without any regularity assumptions on the map. We prove that every compact system has a chain-recurrent point and a closed, invariant, chain-transitive subsystem. The proofs do not rely on the Axiom of Choice.
\end{abstract}



\end{frontmatter}



\section{Introduction}\label{_1}
Chain-recurrence is arguably the most important concept in topological dynamics. On one hand, it has strong connections with applications. Any numerical/experimental evaluation of the time-evolution of a system has nonzero round-off/observational errors bounded by some $\epsilon>0$, so what it really provides is an $\epsilon$-chain rather than an orbit; a theoretical understanding of $\epsilon$-chains is thus crucial for linking abstract dynamical results to what can be in fact computed or measured. On the other hand, the concept of chain-recurrence has deep consequences for our theoretical understanding of long-term behaviors of dynamical systems, underpinning in particular C. Conley's decomposition results (\cite{co78}) and E. Akin's characterization of attractors (\cite{kurka03}, p. 82). 

The aim of this work is to show that chain-recurrence and chain-transitivity retain considerable significance even without assuming that the evolution map is continuous. The study of the dynamical properties of maps with a ``substantial"  
discontinuity set (e.g., dense, residual, or of positive measure) is a relatively recent but quite active research area, with connections with substitutions theory (\cite{DCIP,DCF}), topological entropy (\cite{Pawlak1,Pawlak2}), ergodic theory (\cite{CDC,Ciklova}) and functional analysis (\cite{Steele1,Steele2}).
In the present work, we show that key existence theorems involving chain recurrence can be obtained, not merely when relaxing continuity to weakened forms of regularity (such as quasi-continuity, Baire class 1 or a.e. continuity), but also for utterly arbitrary maps. 

Let us recall a well-known fact: consider a dynamical system $(X,f)$, with $X$ a metrizable space. When $X$ is compact and $f$ is continuous, a classical result by Birkhoff (\cite{birkh}) ensures that there is at least one $x\in X$ that revisits each of its neighborhoods within a bounded number of iterations, that is there exists an almost-periodic point. If we drop the continuity assumption, this is obviously false in general. 

A case in which one is forced to consider the dynamics of completely general maps, and in fact one of the motivations for the present work, is that of quantum dynamics in presence of arbitrary collapse events (this direction is developed in \cite{DFG}). 
\\
In an open quantum system, suppose to follow a single branch of the dynamics, that is a map that lets any state evolve (possibly through collapse events) in a way that is physically admissible and compatible with the evolution of every other state (as done for instance in \cite{QCD} and many other subsequent works). At each time unit, either no collapse occurs (so one applies a unitary operator $U$) or a projective collapse $\{P_i\}$ on the $i$-th eigenvector of the chosen observable occurs. If one fixes a particular outcome itinerary $\omega=\{i_0,i_1,i_2,\dots\}$ that is physically admissible along the realized trajectory (in particular, each  outcome $i_k$ has nonzero Born weight for the state undergoing collapse), the resulting evolution defines a  map $$T_\omega:X\to X$$ (where $X$ is the state space) that has, in general, no reason to be continuous anywhere.
Indeed, arbitrarily close states generally yield post–selected states in different eigenspaces, according to the assumed (in von Neumann's postulates of quantum mechanics, \cite{vn__}) unpredictability of the collapsed state in a single event. Consequently, the hypotheses behind classical recurrence results (such as Poincar\'e/Birkhoff recurrence theorems \cite{poinc_, birkh}) are violated on the observed system. To obtain recurrence results, we need to i) relax the recurrence from almost-periodicity to a suitably generalized form of recurrence, and ii) address the problem for completely arbitrary discrete-time quantum evolutions. 
The abstract version of this problem, if the state space $X$ is compact, is precisely our starting point in the present work.

More generally, the results proved herein clarify the respective role of compactness, chain-structure and regularity in the emergence of recurrence phenomena. This can be summarized by saying that the constraints imposed by compactness and chain relations prevail over the potentially unbounded set-theoretic complexity implied by arbitrary maps.

\vspace{0.2cm}

Searching for a suitable candidate generalization of almost-periodicity, let us review the well-known  \emph{topological dynamical relations} $\mathcal{O},\mathcal{R},\mathcal{N},\mathcal{C}\subseteq X^2$, as introduced by E. Akin (\cite{Akin}). 

For $d:X^2\to\mathbb{R}_0^+$ a metric compatible with the topology on $X$, the topological dynamical relations can be defined as follows (see for instance \cite[Def. 2.2, p. 47]{kurka03}):
\begin{defi}\label{toprel}
  \begin{enumerate}
    \text{        }
        \item Orbit relation: 
        
        $x\,\mathcal{O}\,y$ iff $\exists\, k\in\mathbb{N}\,|\,f^k(x)=y$;
        \item Recurrence relation: 
        
        $x\,\mathcal{R}\,y$ iff $\forall \epsilon>0\, \exists\, k\in\mathbb{N}\,|\,f^k(x)\in B_\epsilon(y)$;
        \item Non-wandering relation:
        
        $x\,\mathcal{N}\,y$ iff $\forall \epsilon>0\,\exists\, z\in B_\epsilon(x)\,\exists\,k\in\mathbb{N}\,|\,f^k(z)\in B_\epsilon(y)$;
        \item Chain-recurrence relation: 
        
        $x\,\mathcal{C}\,y \text{ iff }\forall \epsilon>0\,\exists\, n\in\mathbb{N},\,\exists\, x_0,\dots,x_n\,| x_0=x, x_n=y,\,\\  d(f(x_i),x_{i+1})<\epsilon\ \ (i=0,\dots, n-1)$.
    \end{enumerate}
\end{defi}
Notice that all the relations, although defined in metric terms, are independent of the metric. 

By compactness, every infinite orbit has an accumulation point, which is in fact non-wandering, in the sense that it verifies $x\,\mathcal{N}\,x$. Indeed, any open neighborhood of the accumulation point $x$ of an infinite orbit $\mathcal{O}(y)$ contains two points of type $f^k(y)$, $f^h(y)$ for some distinct natural numbers $h,k$, and thus the claim in Def.\ref{toprel}-3 is verified with $y=x$.
Since finite orbits contain periodic (hence non-wandering) points, we have the following, easy
\begin{fact}\label{fact_}
If $X$ is compact, the non-wandering set $\{x\in X|x\,\mathcal{N}\,x \}$ of $(X,f)$ is nonempty, independently of any assumption on $f$.
\end{fact}
However, the significance of this result is problematic if  we cannot guarantee that $f$ is continuous at the non-wandering point $x$.  Specifically, the statement $x\,\mathcal{N}\,x$
 is independent of the value of $f$ at $x$, so it provides insight into the dynamics of $f$ only when the behavior of $f$ at $x$ can be inferred from its values at sufficiently close points — that is, when $f$ is continuous. 

 A consequence of this is that, without regularity assumptions, in general we have $\mathcal{N}\not\subseteq\mathcal{C}$, so that, in particular, it does \emph{not} directly follow from Fact$^*$ that every compact system has a chain-recurrent point. 
Luckily, it is easy to modify slightly the definition of $\mathcal{N}$ introducing a relation $\widetilde{\mathcal{N}}$ such that: 
\begin{itemize}
    \item $\widetilde{\mathcal{N}}=\mathcal{N}$ in case $f$ is continuous (if $X$ has no isolated points);
    \item the points in $\widetilde{\mathcal{N}}$-relation with $x$ depend, in general, on the value of $f$ at $x$;
    \item we have $\widetilde{\mathcal{N}}\subseteq\mathcal{C}$ also without continuity.
\end{itemize} 
If $f$ is continuous, indeed, it does not matter whether we make the $\epsilon$-correction before or after the first iteration of the map, so in the continuous case the relation $\widetilde{\mathcal{N}}$ defined below is just equivalent to $\mathcal{N}$:
\begin{defi}\label{Ntilde}
    For $x,y\in X$, we write $x\,\widetilde{\mathcal{N}}\, y$ if for every $\epsilon > 0$ there exist $k\in\mathbb{N}_0$ and $z\in B_\epsilon(f(x))$ such that $f^k(z)\in B_\epsilon(y)$.
\end{defi}
In case of a discontinuous map, Def.\ref{Ntilde} is not anymore equivalent to Def.\ref{toprel}-3. In fact, the new version fits better with the other topological dynamical relations. In particular, the statement $x\,\mathcal{A}\,x$ \textit{depends} on the behavior of $f$ at $x$ when $\mathcal{A}$ is any of the relations $\mathcal{O}$, $\mathcal{R}$, $\widetilde{\mathcal{N}}$, $\mathcal{C}$, and we have always $\mathcal{O}\subseteq\mathcal{R}\subseteq\widetilde{\mathcal{N}}\subseteq \mathcal{C}$ (this is graphically represented in the scheme of Fig.\ref{fig_1}), which resembles what happens in the continuous case, when $\mathcal{N}=\widetilde{\mathcal{N}}\subseteq \mathcal{C}$. 

\begin{figure}[H]\label{fig_1}
    \centering
    \includegraphics[width=1\linewidth]{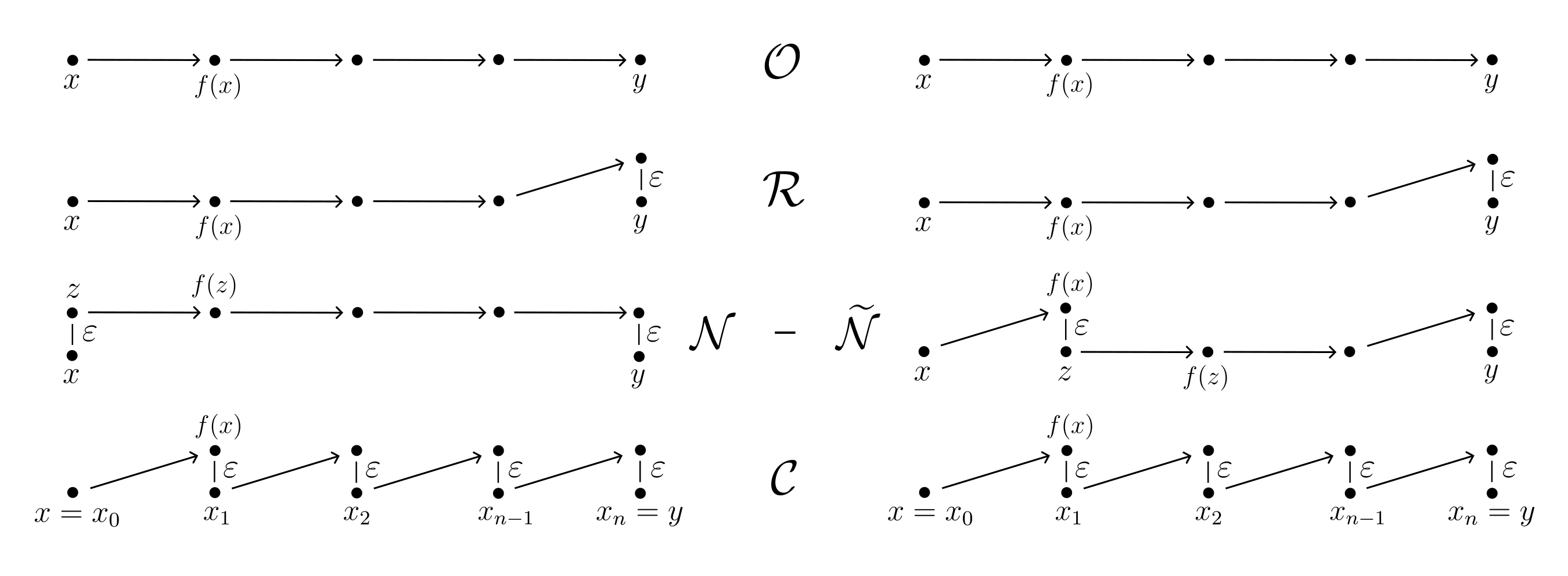}
    \caption{Topological dynamical relations. It is graphically emphasized that, if we replace $\mathcal{N}$ by $\widetilde{\mathcal{N}}$, additional $\epsilon$-corrections are allowed whenever one goes down in the table (right).}
    \label{fig:top_rel}
\end{figure}

Let us now come back to our original problem: the existence of non-wandering points. If we do not assume continuity, it is false that a compact system has always a $\widetilde{\mathcal{N}}$-recurrent point. Indeed, an elementary counterexample to the claim that every compact metrizable dynamical system has a point $x$ such that $x\, \widetilde{\mathcal{N}}\,x$ is given by $([0,1],f)$ where $f$ is defined as: 
\begin{equation}\label{exa_nonr}
f(x)=
\begin{cases}
\vspace{2mm}
\frac 3 4 &\quad \text{ if }x=0\\
\vspace{2mm} 
x(x+\frac 1 2) &\quad \text{ if } x\in (0,\frac 1 2)\\
\vspace{2mm}
\frac 1 4 &\quad \text{ if }x=\frac 1 2\\
\vspace{2mm}
\frac 1 2 x +\frac 1 4 &\quad \text{ if } x\in (\frac 1 2,1]
\end{cases}
\end{equation}
Since all the topological relations $\mathcal{O},\mathcal{R},\widetilde{\mathcal{N}}, \mathcal{C}$ are sensitive to what happens pointwise, one may think that there is little hope of finding any relevant recurrence phenomenon if the map is completely general. This impression is strengthened by the fact that, in the example above, $\widetilde{\mathcal{N}}$ is ruled out by an interval map with just two discontinuity points. If one allows arbitrary maps, the images of $\epsilon$-balls, and in general of open sets, can have arbitrary high (descriptive) set-theoretic complexity, so it seems a priori difficult that any topologically characterized recurrence phenomenon is forced to emerge.

Remarkably, this is not the case: the chain-recurrence relation $\mathcal{C}$ is robust enough, and the compactness hypothesis is sufficiently strict, to ensure key existence results without requiring assumptions on the map. Indeed we will show that: 
\begin{enumerate}
    \item every compact system has a chain-recurrent point; 
    \item every compact system has a chain transitive subsystem;
    \item every point in a compact system is in chain relation with a chain-recurrent point;
    \item every point in a compact system is in chain relation with a chain transitive subsystem.
\end{enumerate} 
The proofs are constructive, in the sense that they do not require the Axiom of Choice (see Remark~\ref{rem_ax_choice}). 

Notice that points 2. and 4. above are respectively stronger than 1. and 3., but their proof builds on their weaker versions, and thus for clarity we preferred to present them as separate, warm-up results.

\section{Settings}

We say that $(X,f)$ is a compact dynamical system (or simply a dynamical system) if $X$ is a compact, metrizable space and $f$ is a self-map of $X$.

For $x\in X$, we indicate by $\mathcal{O}(x)$ the orbit of $x$, that is the set $\mathcal{O}(x)=\{f(x),f^2(x),\ldots\}$. We set $\mathcal{O}^0(x):=\{x\}\cup \mathcal{O}(x)$. 
We will indicate by $\mathbb{N}$ and $\mathbb{N}_0$, respectively, the set of positive and non-negative integers.
Let us recall the notion of $\epsilon$-{\em chain} or $\epsilon$-{\em pseudo-orbit} (\cite[p. 48]{kurka03}).
\begin{defi}
Let $(X,f)$ be a dynamical system and $d:X^2\to\mathbb{R}_0^+$ be a metric compatible with the topology on $X$. Given two points $x,y\in X$ and $\epsilon>0$, an $\epsilon$-\emph{chain} from $x$ to $y$ is a finite set of points $x_0,x_1,\ldots,x_n$ in $X$, with $n\ge 1$, such that
\begin{enumerate}
\item[i)] $x_0=x$ and $x_n=y$,
\item[ii)] $d(f(x_i),x_{i+1})<\epsilon$ for every $i=0,1,\ldots n-1$.
\end{enumerate}
\end{defi}
The {\it chain relation} $\mathcal{C}\subseteq X^2$, the binary relation mentioned in the Introduction, can be equivalently defined as follows (\cite[Def. 2.2, p. 47]{kurka03}): given $x,y\in X$, 
\[
x\,\mathcal{C}\, y\iff \forall \epsilon>0\text{ there exists an $\epsilon$-chain from $x$ to $y$}.
\]
Let $\mathcal{A}\subseteq X^2$ be a binary relation on $X$. For $M\subseteq X$ and $y\in X$, we write $M\,\mathcal{A}\, y$ if $x\,\mathcal{A}\, y$ for every $x\in M$. A point $x\in X$ is called {\it chain-recurrent} if $x\,\mathcal{C}\,x$. 
The set of all the chain-recurrent points of a dynamical system $(X,f)$ is denoted by $CR(X,f)$.
We write simply $CR_f$ when the space $X$ is clear from the context.

Let us also recall the notion of strong chain-recurrence introduced by Easton \cite{easton}.

\begin{defi}
Let $(X,f)$ be a dynamical system and $d:X^2\to\mathbb{R}_0^+$ be a metric compatible with the topology on $X$.
Given two points $x,y\in X$ and $\epsilon>0$, we say that a finite sequence of points $x_0,x_1,\ldots,x_n $ of $X$, with $n\in\mathbb{N}$, is a \emph{strong $(\epsilon,d)$-chain} from
$x$ to $y$ if 
\begin{enumerate}
    \item[i)] $x_0 = x$ and $x_n = y$,
    \item [ii)] $ \sum_{i=0}^{n-1} d(f(x_i),x_{i+1}) <\epsilon$.
\end{enumerate}
The \emph{strong chain relation} $\mathcal{SC}_d\subseteq X^2$ is the binary relation defined as follows: given $x, y \in X$,
\[
x\, \mathcal{SC}_d\, y \iff \forall \epsilon>0\text{ there exists a strong $(\epsilon,d)$-chain from $x$ to $y$}.
\]
A point $x\in X$ is called \emph{strong chain-recurrent} if $x\,\mathcal{SC}_d\,x$. Let $\mathcal{SCR}_d(f)$ be the set of all the strong chain-recurrent points of $(X,f)$. We write simply $\mathcal{SCR}_d$ when the map $f$ is clear from the context.
\end{defi}

The relation $\mathcal{SC}_d$ depends, in general, on the metric $d$ (while $\mathcal{C}$ can be defined in purely topological terms if the space is compact). Intersecting over all topological equivalent metrics on $X$, we obtain the generalized recurrent set $GR(f)$ introduced by Auslander for flows (\cite{Auslander}) and then extended to maps (\cite{Akin,AkinAus,FathiPageault}). More precisely, we give the following
\begin{defi}
    The \emph{generalized recurrent set}, denoted by $GR(f)$, is defined as
    $$
    GR(f)=\bigcap_{d} \mathcal{SCR}_d(f),
    $$
    where the intersection is over all metrics $d$ compatible with the topology of $X$.
\end{defi}
The following inclusions hold without any assumptions on $f$: $$GR(f)\subseteq \mathcal{SCR}_d(f)\subseteq CR_f.$$

\begin{defi}
Let $S$ be a subset of a topological space $X$. The \emph{derived set} of $S$, denoted by $S'$, is the set of the limit points $x$ of $S$, that is, the set of points $x$ such that for every neighborhood $U$ of $x$ we have that 
$S\cap (U\setminus\{x\})\neq \emptyset$.
\end{defi}

\begin{defi}\label{def_inv}
    Let $(X,f)$ be a dynamical system, and $S\subseteq X$. We say that the set $S$ is \emph{$f$-invariant} (or simply invariant) if $f(S)\subseteq S$.
\end{defi}
\begin{defi}\label{def_subsy}
    Let $(X,f)$ be a dynamical system. For $Y\subseteq X$, we say that $(Y,f\vert_Y)$ is a \emph{subsystem} of $(X,f)$ if $Y$ is closed and invariant. 
\end{defi}

\begin{defi}\label{def_chain_trans}
 Let $(X,f)$ be a dynamical system, and $M\subseteq X$. We say that $M$ is \emph{chain-transitive} ($\mathcal{C}$-\emph{transitive}) if for every $x,y\in M$ we have that $x\,\mathcal{C}\, y$.
\end{defi}

\begin{defi}\label{def_strongchain_trans}
    Let $(X,f)$ be a dynamical system and $d:X^2\to\mathbb{R}_0^+$ be a metric compatible with the topology on $X$.
    We say that $M\subseteq X$ is \emph{strong chain-transitive} ($\mathcal{SC}_d$-\emph{transitive}) if for every $x,y\in M$ we have that $x\,\mathcal{SC}_d \, y$.
\end{defi}
The following is a strengthening of chain-transitivity which will prove useful later on.
\begin{defi}\label{def_chain_compl}
  Let $(X,f)$ be a dynamical system, and $M\subseteq X$. We say that $M$ is \emph{internally chain transitive} (see e.g. \cite{Hirsch,Garay,surface}, who use the concept in the case of continuous maps and flows) if for every $x,y\in M$ and for every $\epsilon>0$ there exists an $\epsilon$-chain form $x$ to $y$ whose points belong to $M$. 
\end{defi}
The concept of internal chain transitivity is related to what is called, in case of continuous maps, the \emph{restriction property} for topological dynamical systems (see for instance \cite{Alongi}, Prop. 2.7.21), which ensures that, setting $CR(X,f)=S$, we have $$CR(S,f\vert_S)=S.$$ This means that, for every $x\in S$, the $\epsilon$-chains from $x$ to itself can be built with points belonging to $S$. Notice, however, that this does not imply, even in the continuous case, that two distinct points $x,y\in S$ that are in chain relation can necessarily be connected by $\epsilon$-chains built with points of $S$, as is shown elementarily by $([0,1],f)$ with $f=x^2$. Indeed, here, $1\,\mathcal{C}\, 0$ but the $\epsilon$-chains must of course leave the chain recurrent set $\{0,1\}$.

\section{Results}
\begin{thm}
\label{thm1}
Let $(X,f)$ be a compact dynamical system. Then $CR_f\ne\emptyset$.
\end{thm}
\begin{proof}
If there is $x\in X$ such that $|\mathcal{O}(x)|<\aleph_0$, then at least one point $y$ in the orbit of $x$ is periodic, so that $y\,\mathcal{C}\,y$. We can therefore assume that every point in $X$ has an infinite orbit.
Let $d:X^2\to\mathbb{R}_0^+$ be a metric compatible with the topology on $X$.  
\begin{itemize}
\item[\textbf{Step 1}] Pick $x\in X$ and set $x_0:=x$. For $k\in\mathbb{N}$, set $x_k:=f(x_{k-1})$. Pick $y\in \mathcal{O}(x)'$ (which is nonempty by compactness) and set $x_\omega:=y$.
Let us now define a set of points by transfinite induction as follows. 
Let $\alpha>\omega$ be an ordinal number. 
Assume that $x_\beta$ has been defined for every $\beta<\alpha$ and set: $$\{x_\eta\}_{\beta\le\eta<\alpha}:=S_{\beta,\alpha}(x)\quad , \quad S_{\alpha}(x):=S_{0,\alpha}(x).$$
\begin{enumerate}
\item If $\alpha$ is a successor ordinal, set $x_\alpha:=f(x_{\alpha-1})$. 
\item If $\alpha$ is a limit ordinal, then, recalling that every orbit in $X$ has infinite cardinality, we have $|S_{\beta,\alpha}(x)|\ge\aleph_0$ for every $\beta<\alpha$, so $\left(S_{\beta,\alpha}(x)\right)'$ is closed and, by compactness, nonempty. Therefore,
\begin{equation}S^*_\alpha(x):=\bigcap_{\beta<\alpha}\left(S_{\beta,\alpha}(x)\right)'
\end{equation}
is nonempty as well (see for instance \cite[Th. 26.9, p. 169]{munkres2013}). Pick then any $y\in S^*_\alpha(x)$ and set $x_\alpha:=y$.
\end{enumerate}

\item[\textbf{Step 2}] 
For every ordinal $\alpha$, we have 
\begin{equation}
\label{1}
x_\beta \,\mathcal{C}\, x_\eta\quad \text{ whenever }\quad  0\le\beta<\eta\le\alpha.
\end{equation} 
Let us prove this by induction on $\alpha$. 
If $\alpha=1$, the claim is obviously true. Assume then that $\alpha>1$ and that claim \eqref{1} is true for every $\eta<\alpha$. 
If $\alpha$ is a successor ordinal, then $x_\alpha=f(x_{\alpha-1})$, and thus, for $0\le\beta<\eta\le\alpha$, we have 
$$
x_\beta\,\mathcal{C}\,x_{\alpha-1}\implies x_\beta\,\mathcal{C}\,x_{\alpha}.
$$ 
Assume that $\alpha$ is a limit ordinal and take $\beta<\alpha$ and $\epsilon>0$. 
Since $x_\alpha\in S^*_\alpha(x)$, there is $\beta<\eta<\alpha$ such that $d(x_\eta,x_\alpha)<\frac{\epsilon}{2}$. 
By the induction hypothesis, we have $x_\beta\,\mathcal{C}_{\frac{\epsilon}{2}}\,x_\eta$, so by the triangle inequality $x_\beta\,\mathcal{C}_\epsilon\, x_\alpha$, and since $\epsilon$ was arbitrary, we have $x_\beta\,\mathcal{C}\,x_\alpha$.

\item [\textbf{Step 3}]
Now, if for every $\eta<\alpha$ we have $x_\eta\notin S_\eta(x)$, then $x_\beta\ne x_\gamma$ for every distinct $\beta,\gamma<\alpha$, so that the application 
\begin{equation}\label{eq_y}
\eta\mapsto x_\eta\quad\quad (0\le\eta<\alpha)
\end{equation}
is a bijection between $\alpha=\{\eta:0\le\eta<\alpha\}$ and $S_\alpha(x)$, so
we have $|S_\alpha(x)|=|\alpha|$. 
Assuming $x_\alpha\notin S_\alpha(x)$ for every ordinal $\alpha$, by Hartogs' Lemma (\cite{hartogs}) there would be an ordinal $\lambda$ so large that $|S_\lambda(x)|>|X|$, in the sense that there is no injection from $S_\lambda(x)$ into $X$. But this is impossible since $S_\alpha(x)\subseteq X$. 
Therefore, $x_\alpha\in S_\alpha(x)$ for some ordinal $\alpha$, so that: 
\begin{equation}\label{eq_x}
x_\alpha=x_\beta\text{ for some } \beta<\alpha.
\end{equation}
This implies, using claim \eqref{1} with $\eta=\alpha$, that $x_\alpha$ is chain-recurrent. 
\end{itemize}
\end{proof}
\begin{rem}\label{rem_nest}
Notice that, since in Step 2 we built every chain using only points indexed by ordinals, for every $\beta<\eta\le\alpha$ and every $\epsilon>0$, there is an $\epsilon$-chain from $x_\beta$ to $x_\eta$ whose points  belong to $S_{\beta,\eta+1}(x)$.   
\end{rem}
\begin{rem}\label{rem_ax_choice}
The argument in the previous proof does not use the Axiom of Choice. In particular, the selection at limit ordinals is possible in ZF+DC (\cite{Kura_}), while Hartogs' Lemma holds in ZF (\cite{moerdijk}, Chapter 1). The same holds for the other existence results in this work. 

Since a compact space has at most cardinality continuum (see \cite[Corollary 3.1.30]{Engelking}), assuming the Generalized Continuum Hypothesis we can take, in the previous proof, $\lambda\le \aleph_2=2^{\mathfrak{c}}$.
\end{rem}

If we replace the chain relation $\mathcal{C}$ with the relation $\widetilde{\mathcal{N}}$, the argument in the previous proof is no longer applicable. In fact, in Step 2, for $\alpha=1$ claim \eqref{1} is also true by replacing $\mathcal{C}$ with $\widetilde{\mathcal{N}}$, so we can assume as inductive hypothesis that,  for $\alpha>1$, claim \eqref{1} is true, with $\widetilde{\mathcal{N}}$ instead of $\mathcal{C}$, for every $\eta<\alpha$. However, when $\alpha$ is a successor ordinal, even if we have $x_\beta\,\widetilde{\mathcal{N}}\, x_{\alpha-1}$ for $0\le \beta <\eta \le \alpha$, it is not true, in general, that $x_\beta\,\widetilde{\mathcal{N}}\, f(x_{\alpha-1})$ when $f$ is discontinuous at $x_{\alpha-1}$. This of course is consistent with the example given in Eq. \eqref{exa_nonr} in the Introduction. 

The following result is an easy consequence of the previous theorem.
\begin{thm}
\label{cor1}
 Let $(X,f)$ be a compact dynamical system. Then for every $x\in X$ there is $y\in X$ such that $y$ is chain-recurrent and $x\,\mathcal{C}\, y$.
\end{thm}
\begin{proof}
In the previous argument, and specifically in Eq.\eqref{eq_x}, take $y=x_\alpha$. Since $x=x_0$ was arbitrary, the result follows.
\end{proof}
Notice that, when the map is continuous, the result above can be strengthened by replacing the last relation $x\,\mathcal{C}\,y$ by $x\,\mathcal{R}\,y$, and it is a consequence of the existence of a complete Lyapunov function (a seminal work is \cite{co78}). 

\begin{thm}
\label{th_strongchain}
Let $(X,f)$ be a compact dynamical system. Then $GR(f)\ne\emptyset$.
\end{thm}
\begin{proof}
Again we can assume that every orbit in $X$ is infinite (otherwise the result follows immediately). Let $\mathcal{D}$ be the set of all metrics compatible with the topology on $X$, and pick $d\in\mathcal{D}$.
Pick $x\in X$ and set $x_0:=x$. For $k\in\mathbb{N}$, set $x_k:=f(x_{k-1})$.  
For any ordinal $\alpha$, we define the sets $S_\alpha(x)$ and $S^*_\alpha(x)$ as done in Step 1 of the proof of Theorem \ref{thm1}. Notice that this construction depends only on the topology and on the map $f$, but does not depend on the metric $d$.

Following the steps in the proof of  Theorem \ref{thm1}, we want to prove by induction on $\alpha$ that 
\begin{equation}
\label{sc_1}
x_\beta \,\mathcal{SC}_d\, x_\eta\quad \text{ whenever }\quad  0\le\beta<\eta\le\alpha.
\end{equation} 
If $\alpha=1$, the claim is obviously true; therefore, we assume that claim \eqref{sc_1} is true for every $\eta<\alpha$. 
If $\alpha$ is a successor ordinal, then $x_\alpha=f(x_{\alpha-1})$, and thus, for $0\le\beta<\eta\le\alpha$, we have 
$$
x_\beta\,\mathcal{SC}_d\,x_{\alpha-1}\implies x_\beta\,\mathcal{SC}_d\,x_{\alpha}.
$$ 
Assume that $\alpha$ is a limit ordinal and take $\beta<\alpha$ and $\epsilon>0$. 
Since $x_\alpha\in S^*_\alpha(x)$, there is $\beta<\eta<\alpha$ such that $d(x_\eta,x_\alpha)<\frac{\epsilon}{2}$. 
By the induction hypothesis, there is a strong $(\frac \epsilon 2,d)$-chain $x_\beta=x_0,x_1,\ldots,x_n=x_\eta$ from $x_\beta$ to $x_\eta$.
By the triangle inequality, we have: 
$$
\sum_{i=0}^{n-2}d(f(x_i),x_{i+1})+d(f(x_{n-1}),x_\alpha)
\le
\sum_{i=0}^{n-1}d(f(x_i),x_{i+1})+d(x_\eta,x_\alpha)<
\epsilon.
$$
Therefore, the points $x_\beta=x_0,x_1,\ldots, x_{n-1},x_\alpha$ form a strong $(\epsilon,d)$-chain from $x_\beta$ to $x_\alpha$. 
Since $\epsilon$ was arbitrary, it follows that $x_\beta\,\mathcal{SC}_d\,x_\alpha$. 
Then, by an identical argument as in Step 3 above (with \eqref{sc_1} in place of \eqref{1}), we have that there is some ordinal $\alpha$ such that $x_\alpha\in \mathcal{SCR}_d$. 

Since claim \eqref{sc_1} is valid for every choice of the metric $d$ (compatible with the topology), 
we have 
$$
x_\alpha \in\bigcap_{d\in\mathcal{D}} \mathcal{SCR}_d(f)=GR(f).
$$ 
\end{proof}

We prove now that every compact system has a chain-transitive subsystem. For this we need a preliminary lemma, which deepens our understanding of the construction in the proof of Theorem \ref{thm1}.

\begin{lem}\label{thm2}
    Let $(X,f)$ be a compact dynamical system. Then $X$ has an invariant, $\mathcal{C}$-transitive subset.
\end{lem}
\begin{proof}
If some point is periodic, the periodic orbit is clearly an invariant, $\mathcal{C}$-transitive subset of $X$. 
We can then assume that every orbit is infinite.
We will use the notation defined in the proof of Theorem \ref{thm1}. 
For every $\alpha$ limit ordinal and $\beta\le\alpha$ we set
$$
\Gamma_{\beta,\alpha}:=\{\beta\}\cup \{ \beta< \eta<\alpha : \text{$\eta
$ is a limit ordinal}\}.
$$ 
Pick $x\in X$. 
Let $\lambda$ be an ordinal such that $|\lambda|>|X|$ and $S_\lambda(x)$ a set defined by means of the transfinite procedure described in Step 1 of the proof of Theorem \ref{thm1}. Let $\alpha\le\lambda$ be the least ordinal greater than 0 such that $x_\alpha\in S_\alpha(x)$. 
Then there is an ordinal $0\le \beta<\alpha$ such that $x_\alpha=x_\beta$ and $x_\alpha \in CR_f$.

If $\alpha$ is a limit ordinal, then the set $S_{\beta,\alpha}(x)$ can be written as:
\[
S_{\beta,\alpha}(x)=\bigcup_{\eta\in\Gamma_{\beta,\alpha}} \mathcal{O}^0(x_\eta),
\]
that is, it is a union of orbits and therefore is invariant. Analogously, if $\alpha$ is a successor ordinal, that is $\alpha=\gamma +k$ for some limit ordinal $\gamma\ge \beta$ and $k\in\mathbb{N}$ the set $S_{\beta,\alpha}(x)$ can be written as:
\[
S_{\beta,\alpha}(x)=\bigcup_{\eta\in\Gamma_{\beta,\gamma}} \mathcal{O}^0(x_\eta) \cup 
\bigcup_{h=0}^{k-1} f^h(x_\gamma)
,
\]
which is again invariant, since in this case $f^k(x_\gamma)=x_\alpha=x_\beta$. 
In both cases, since claim \eqref{1} holds, that is, $x_\beta \,\mathcal{C}\,x_\eta$ for every $0\le\beta<\eta\le \alpha$ and $x_\alpha=x_\beta$, by the transitivity of $\mathcal{C}$, it follows that $S_{\beta,\alpha}(x)$ is $\mathcal{C}$-transitive.
\end{proof}

\begin{rem}\label{rem_chain_compl}
Notice that by Remark \ref{rem_nest} and since $x_\alpha=x_\beta$ it follows that for every $z,y\in S_{\beta,\alpha}(x)$ and for every $\epsilon>0$ there exists an $\epsilon$-chain from $z$ to $y$ whose points belong to $S_{\beta,\alpha}(x)$. This is equivalent to say that the set $S_{\beta,\alpha}(x)$ is internally chain transitive.
\end{rem}

\begin{rem}
  Notice that by claim \eqref{sc_1} and since $x_\alpha=x_\beta$, by the transitivity of $\mathcal{SC}_d$, it follows that $S_{\beta,\alpha}(x)$ is $\mathcal{SC}_d$-transitive.
\end{rem}

\begin{thm}\label{thmsubsyst}
    Let $(X,f)$ be a compact dynamical system. Then $(X,f)$ has a $\mathcal{C}$-transitive subsystem.
\end{thm}
\begin{proof}
Let us first of all observe that, to prove the statement, we need to show that there is a closed, invariant and internally chain transitive subset of $X$. In particular, the last property ensures that we obtained an object which is a chain-transitive \emph{subsystem} of $(X,f)$, and not just a closed, invariant and chain-transitive subset.
We will use the notation defined in the proof of Theorem \ref{thm1}: for $x \in X$ and for ordinals $\zeta<\iota$, let $S_{\zeta,\iota}(x)$ and $S_{\zeta}(x)=S_{0,\zeta}(x)$ be sets defined as in Step 1 of the proof of Theorem \ref{thm1}.

Suppose, towards a contradiction, that $(X,f)$ does not have a chain-transitive subsystem. 
In particular, we can assume that every orbit is infinite.
Let $x^0\in X$. By Theorem \ref{thm1}, there exists a least ordinal $\alpha_0>0$ such that $x^0_{\alpha_0}\in S_{\alpha_0}(x^0)$, and so $x^0_{\alpha_0}=x^0_{\beta_0}$ for some $\beta_0<\alpha_0$. 
By Lemma \ref{thm2} and Remark \ref{rem_chain_compl}, the set $S_{\beta_0,\alpha_0}(x^0)$ is an invariant, internally chain transitive subset. 
Since $S_{\beta_0,\alpha_0}(x^0)$ cannot be closed, there exists some $y\in (S_{\beta_0,\alpha_0}(x^0))'\setminus S_{\beta_0,\alpha_0}(x^0)$.
By Remark \ref{rem_nest} and by the triangle inequality, we have that for every $x\in S_{\alpha_0}(x^0)$ and for every $\epsilon>0$ there exists an $\epsilon$-chain from $x$ to $y$ whose points belong to $S_{\alpha_0}(x^0)\cup \{y\}$.
If $y\neq x^0$, then we set $x^1:=y$.
Otherwise, if $y=x^0$, it follows that $S_{\alpha_0}(x^0)$ is an invariant, internally chain transitive subset. 
Since it cannot be a closed set, there is some $z\in (S_{\alpha_0}(x^0))'\setminus S_{\alpha_0}(x^0)$, and we set $x^1:=z \,(\neq x^0)$. 
We then have $S_{\alpha_0}(x^0)\, \mathcal{C}\, x^1$, and for every $x\in S_{\alpha_0}(x^0)$ and for every $\epsilon>0$ there exists an $\epsilon$-chain from $x$ to $x^1$ whose points belong to $S_{\alpha_0}(x^0)\cup \{x^1\}$.

Summarizing, we have shown that the set $\{x^0,x^1\}$ has the following properties:
\begin{itemize}
    \item[(a$_1$)] there exists a least ordinal $\alpha_0$ such that $x^0_{\alpha_0}=x^0_{\beta_0}$ for some $\beta_0<\alpha_0$ and  $S_{\beta_0,\alpha_0}(x^0)$ is an invariant, internally chain transitive subset;
    \item[(b$_1$)] $S_{\alpha_0}(x^0)\, \mathcal{C}\, x^1$, and for every $x\in S_{\alpha_0}(x^0)$ and for every $\epsilon>0$ there exists an $\epsilon$-chain from $x$ to $x^1$ whose points belong to $S_{\alpha_0}(x^0)\cup \{x^1\}$;
    \item[(c$_1$)] $x^0 \neq x^1$.
\end{itemize}
Let us proceed now by transfinite induction. 
Let $\lambda>1$ be an ordinal number. 
Assume that $x^\gamma$ has been defined for every $0\le \gamma<\lambda$ and that the set $\{x^\gamma\}_{\gamma<\lambda}$ has the following property:
\begin{enumerate}
\item[(a$_\lambda$)] for every $0\le \gamma<\lambda$, there exists a least ordinal $\alpha_\gamma$ such that $x^\gamma_{\alpha_\gamma}=x^\gamma_{\beta_\gamma}$ for some $\beta_\gamma<\alpha_\gamma$ and  $S_{\beta_\gamma,\alpha_\gamma}(x^\gamma)$ is an invariant, internally chain transitive subset.
\end{enumerate}
Moreover, setting for every $\beta<\alpha\le \lambda$  
\[
M_{\beta,\alpha}:=\bigcup_{\beta\le \gamma<\alpha} S_{\alpha_\gamma}(x^\gamma),
\]
assume also the following properties:
\begin{enumerate}
\item[(b$_\lambda$)] For every $\eta<\xi<\lambda$, we have $S_{\alpha_\eta}(x^\eta)\,\mathcal{C}\, x^{\xi}$, and for every $x\in S_{\alpha_\eta}(x^\eta)$ and every $\epsilon>0$ there exists an $\epsilon$-chain from $x$ to $x^\xi$ whose points belong to $M_{\eta,\xi}\cup \{x^\xi\}$;

\item[(c$_\lambda$)] $x^\eta \neq x^\xi$ whenever $\eta\neq \xi$ with $\eta,\xi<\lambda$ . 
\end{enumerate}

We say that a point $z$ verifies property P$_\lambda$ if, for every $\eta<\lambda$, we have $S_{\alpha_\eta}(x^\eta)\,\mathcal{C}\, z$, and for every $x\in S_{\alpha_\eta}(x^\eta)$ and every $\epsilon>0$, there exists an $\epsilon$-chain from $x$ to $z$ whose points belong to $M_{\eta,\lambda}\cup \{z\}$.

Notice that, assuming (b$_\lambda$), to prove property (b$_{\lambda+1}$) it is sufficient to show that the point $x^\lambda$ verifies property P$_\lambda$.
We now proceed with the definition of the point $x^\lambda$.

\begin{itemize}
    \item[Case 1.] $\lambda$ is a successor ordinal.\\
    Consider the set $S_{\beta_{\lambda-1},\alpha_{\lambda-1}}(x^{\lambda-1})$, which is invariant and internally chain transitive by property (a$_\lambda$). Since it cannot be a closed set, there exists some 
    $$
    y\in (S_{\beta_{\lambda-1},\alpha_{\lambda-1}}(x^{\lambda-1}))'\setminus S_{\beta_{\lambda-1},\alpha_{\lambda-1}}(x^{\lambda-1}).
    $$
    The point $y$ verifies property P$_\lambda$.
    Indeed, by Remark \ref{rem_nest}, and by the triangle inequality, we have that for every $x\in S_{\alpha_{\lambda-1}}(x^{\lambda-1})$ and for every $\epsilon>0$ there exists an $\epsilon$-chain from $x$ to $y$ whose points belong to $S_{\alpha_{\lambda-1}}(x^{\lambda-1})\cup \{y\}$. Fix $\eta<\lambda$ and $\epsilon>0$ and take $x\in S_{\alpha_\eta}(x^\eta)$. 
    Then there is an $\epsilon$-chain from $x^{\lambda-1}\in S_{\alpha_{\lambda-1}}(x^{\lambda-1})$ to $y$ whose points belong to $S_{\alpha_{\lambda-1}}(x^{\lambda-1})\cup \{y\}$. 
    By property (b$_\lambda$), for every $\epsilon>0$ there exists an $\epsilon$-chain from $x$ to $x^{\lambda-1}$ whose points belong to $M_{\eta,\lambda-1}\cup \{x^{\lambda-1}\}$.
    Therefore, noting that $M_{\eta,\lambda}=M_{\eta,\lambda-1}\cup S_{\alpha_{\lambda-1}}(x^{\lambda-1})$ and using the transitivity of $\mathcal{C}$, we have that there is an $\epsilon$-chain from $x$ to $y$ whose points belong to $M_{\eta,\lambda}\cup \{y\}$.\\
    Consider the following two cases.
    \begin{itemize}
        \item If $y\notin \{x^\gamma\}_{\gamma<\lambda}$, then we set $x^\lambda:=y$. Then $x^\lambda$ verifies property P$_\lambda$ and (b$_{\lambda+1}$) is satisfied.
        
        \item Assume that $y=x^{\xi_0}$ for some $\xi_0<\lambda$.
        By property (b$_\lambda$) and since $y$ verifies property P$_\lambda$, we have that the set $M_{\xi_0,\lambda}$ is invariant and internally chain transitive. 
        Since this set cannot be closed, there is some $z_1\in (M_{\xi_0,\lambda})'\setminus M_{\xi_0,\lambda}$, so that $z_1\notin \{x^\gamma\}_{\xi_0\le \gamma<\lambda}$. 
        Since $M_{\xi_0,\lambda}$ is internally chain transitive and $z_1\in (M_{\xi_0,\lambda})'$, by the triangle inequality it follows that for every $x\in M_{\xi_0,\lambda}$ there exists an $\epsilon$-chain from $x$ to $z_1$ whose points belong to $M_{\xi_0,\lambda}\cup \{z_1\}$.
        By property (b$_\lambda$) and using the transitivity of $\mathcal{C}$, we have that
        the point $z_1$ verifies property P$_\lambda$.
        
        If $z_1=x^{\xi_1}$ for some $\xi_1<\xi_0$, we repeat the same argument replacing $\xi_1$ with $\xi_0$ and observing that the set $M_{\xi_1,\lambda}$ is invariant and internally chain transitive.

        Since there is no infinitely decreasing sequence of ordinals, we can repeat the previous construction up to a certain  $k\in\mathbb{N}$, after which we must have $z_{k+1}\notin \{x^\gamma\}_{\gamma<\lambda}$.
        Then we set $x^\lambda:=z_{k+1}$ and since it verifies property P$_\lambda$,  property (b$_{\lambda+1}$) is satisfied.
    \end{itemize}
    Consider the sequence $\{x^\gamma\}_{\gamma\le \lambda}$. By construction, property (c$_{\lambda+1}$) is verified, and from Lemma \ref{thm2} and Remark \ref{rem_chain_compl} property (a$_{\lambda+1})$ follows.
    
    \item[Case 2.] $\lambda$ is a limit ordinal.\\
    Since $|\{x^\xi \ |\ \gamma\le \xi <\lambda\}|\ge \aleph_0$ for every $\gamma<\lambda$, by compactness (see \cite[Th. 26.9, p. 169]{munkres2013}) we can pick 
    $$
    y\in \bigcap_{\gamma<\lambda}(\{x^\xi \ |\ \gamma\le \xi <\lambda\})'.
    $$
    The point $y$ verifies property P$_\lambda$.
    Indeed, fix $\eta<\lambda$ and $\epsilon>0$ and take $x\in S_{\alpha_\eta}(x^\eta)$. Then there exists $\eta<\xi<\lambda$ such that $d(x^\xi,y)<\frac \epsilon 2$.
    By property (b$_\lambda$), there exists an $\frac \epsilon 2$-chain from $x$ to $x^\xi$ whose points belong to $M_{\eta,\xi}\cup \{x^\xi\}$. By the triangle inequality, it follows that there exists an $\epsilon$-chain from $x$ to $y$ whose points belong to $M_{\eta,\xi}\cup \{y\}\subseteq M_{\eta,\lambda}\cup \{y\}$.\\
    Consider the following two cases.
    \begin{itemize}
        \item If $y\notin \{x^\gamma\}_{\gamma<\lambda}$, then we set $x^\lambda:=y$. 
        Then, $x^\lambda$ verifies property P$_\lambda$ and (b$_{\lambda+1}$) is satisfied.
        
        \item Assume that $y=x^{\xi_0}$ for some $\xi_0<\lambda$. 
        By property (b$_\lambda$) and since $y$ verifies property P$_\lambda$ we have that the set $M_{\xi_0,\lambda}$ is invariant and internally chain transitive. 
        Since this set cannot be closed, there exists $z_1$ such that 
        $z_1 \in \left(M_{\xi_0,\lambda}\right)'\setminus M_{\xi_0,\lambda}.$
        Therefore, we have that $z_1\notin \{x^\gamma\}_{\xi_0 \le \gamma<\lambda}$. 
        Since $M_{\xi_0,\lambda}$ is internally chain transitive and $z\in (M_{\xi_0,\lambda})'$, by property (b$_\lambda$) and by the triangle inequality it follows that for every $x\in M_{\xi_0,\lambda}$ there exists an $\epsilon$-chain from $x$ to $z_1$ whose points belong to $M_{\xi_0,\lambda}\cup \{z_1\}$.
        By property (b$_\lambda$) and using the transitivity of $\mathcal{C}$, we have that
        the point $z_1$ verifies property P$_\lambda$.
        
        If $z_1=x^{\xi_1}$ for some $\xi_1<\xi_0$, we repeat the same argument replacing $\xi_1$ with $\xi_0$ and observing that the set $M_{\xi_1,\lambda}$ is invariant and internally chain transitive.

        Since there is no infinitely decreasing sequence of ordinals, we can repeat the previous construction up to a certain  $k\in\mathbb{N}$, after which we must have $z_{k+1}\notin \{x^\gamma\}_{\gamma<\lambda}$.
        Then we set $x^\lambda:=z_{k+1}$ and since it verifies property P$_\lambda$,  property (b$_{\lambda+1}$) is satisfied. 
    \end{itemize}
By construction, the sequence $\{x^\gamma\}_{\gamma\le \lambda}$ verifies property (c$_{\lambda+1}$), and from Lemma \ref{thm2} and Remark \ref{rem_chain_compl} property (a$_{\lambda+1})$ follows.
\end{itemize}
Assuming that $(X,f)$ does not have a chain-transitive subsystem, we find that the application
$$
\gamma\mapsto x^\gamma \quad (0\le\gamma<\lambda)
$$
is a bijection between $\lambda$ and the set $\{x^\gamma\}_{0\le\gamma<\lambda}$, so we have that $|\{x^\gamma\}_{0\le\gamma<\lambda}|=|\lambda|$. 
By Hartogs' Lemma (\cite{hartogs}), we can take $\lambda$ so large that $|\{x^\gamma\}_{0\le \gamma<\lambda}|>|X|$, which is a contradiction. 

Therefore, there exists an ordinal $\nu<\lambda$ such that it is impossible to define the point $x^\nu$.
More precisely, we have that, if $\nu$ is a successor ordinal, one cannot define the point $x^\nu$ if one of the following verifies:
\begin{itemize}
    \item The set $S_{\beta_{\nu-1},\alpha_{\nu-1}}(x^{\nu-1})$ is closed, so that
    $$(S_{\beta_{\nu-1},\alpha_{\nu-1}}(x^{\nu-1}))'\subseteq S_{\beta_{\nu-1},\alpha_{\nu-1}}(x^{\nu-1}).$$
    
    \item The set $S_{\beta_{\nu-1},\alpha_{\nu-1}}(x^{\nu-1})$ is not closed and there exists an ordinal $\xi<\nu$ such that the set $M_{\xi,\nu}$ is a closed, invariant and internally chain transitive subset, that is a chain-transitive subsystem.
    \end{itemize}

On the other hand, if $\nu$ is a limit ordinal, then we cannot define the point $x^\nu$ if for every $y\in\bigcap_{\gamma<\nu}(\{x^\xi \ |\ \gamma\le \xi <\nu\})'$, we have that $y\in\{x^\gamma\}_{\gamma<\nu}$. 
This implies that there exists an ordinal $\xi<\nu$ such that $M_{\xi,\nu}$ is a closed, invariant and internally chain transitive subset, that is a chain-transitive subsystem. 
\end{proof}

\begin{figure}[H]\label{fig_2}
    \centering
    \includegraphics[width=1\linewidth]{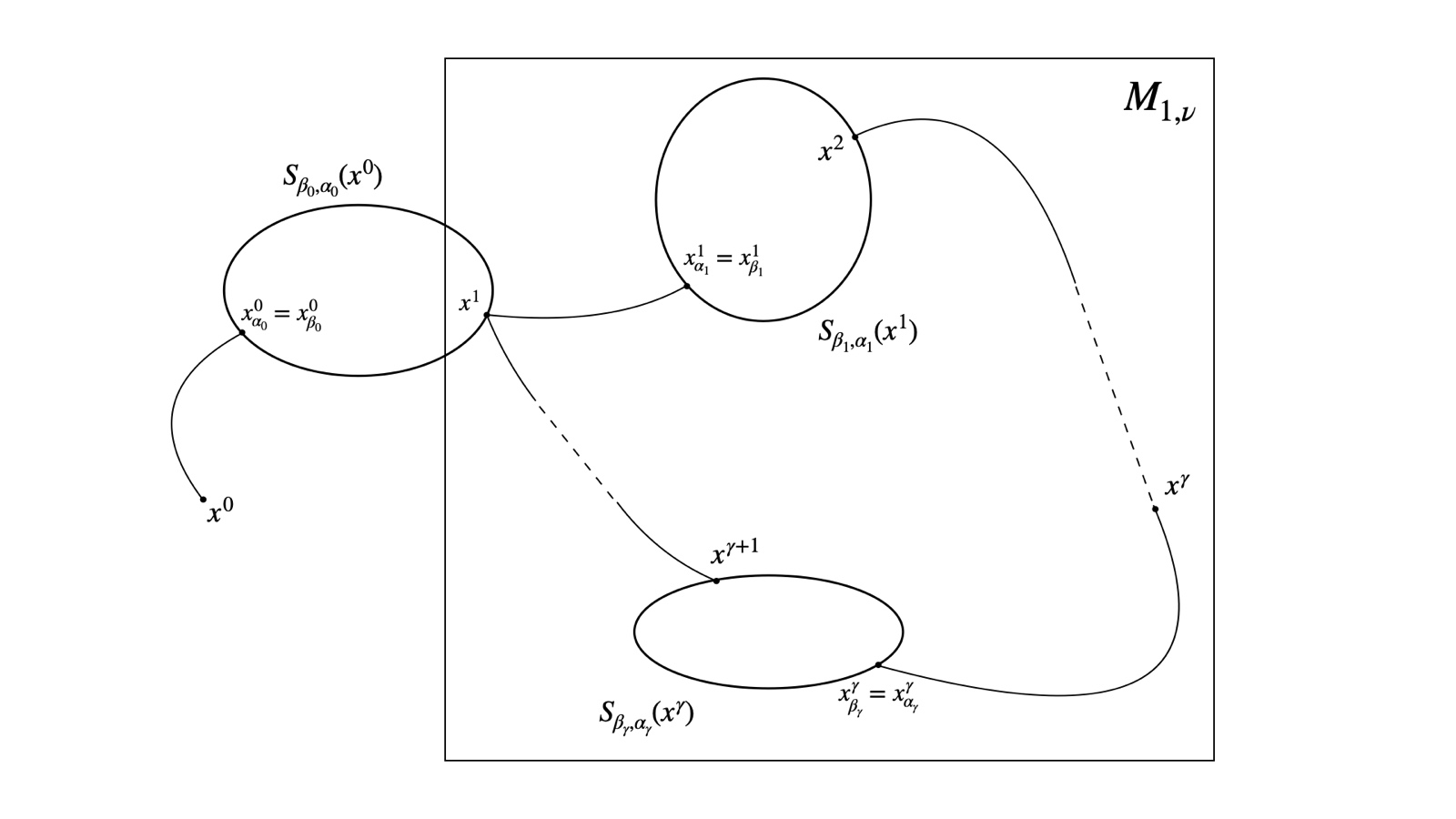}
    \caption{Graphic representation of (a sample case of) the reasoning in the proof of Theorem \ref{thmsubsyst}.}
    \label{fig:subsystem}
\end{figure}
Notice that, assuming $f$ continuous, we can strengthen the previous result  replacing ``chain-transitive" by ``minimal" (see \cite{kurka03}, Proposition 2.20). 

From the previous theorem we can readily deduce the following result.

\begin{thm}
\label{thm_cr}
Let $X$ be a compact dynamical system.
Then every $x\in X$ is in chain-recurrence relation with every point of certain closed, invariant, internally chain transitive subset $S\subseteq X$.
\end{thm}
\begin{proof}
Pick any $x^0\in X$. By Theorem \ref{thmsubsyst} there exists a least ordinal $\nu$ such that the set $\{x^\gamma\}_{\gamma<\nu}$ is defined, while it is impossible to define the point $x^\nu$. 

If $\nu$ is a successor ordinal, then one of the following verifies:
\begin{enumerate}
    \item The set $S_{\beta_{\nu-1},\alpha_{\nu-1}}(x^{\nu-1})$ is a closed, invariant and internally chain transitive subset. \\
    By construction of the set $\{x^\gamma\}_{\gamma<\nu}$ we have $x^0\,\mathcal{C}\, S_{\beta_{\nu-1},\alpha_{\nu-1}}(x^{\nu-1})$ and the claim is verified with $x=x^0$ and $S=S_{\beta_{\nu-1},\alpha_{\nu-1}}(x^{\nu-1})$.
    
    \item The set $S_{\beta_{\nu-1},\alpha_{\nu-1}}(x^{\nu-1})$ is not closed and there exists an ordinal $\xi<\nu$ such that the set $M_{\xi,\nu}$ is closed, invariant and internally chain transitive. Again, by construction, we have $x^0\,\mathcal{C}\,M_{\xi,\nu}$ and the claim is verified with $x=x^0$ and $S=M_{\xi,\nu}$.
    
    \end{enumerate}
On the other hand, if $\nu$ is a limit ordinal and we cannot define the point $x^\nu$, then there exists an ordinal $\xi<\nu$ such that $M_{\xi,\nu}$ is a closed, invariant and internally chain transitive subset. 
Moreover, we have $x^0\,\mathcal{C}\,M_{\xi,\nu}$ and the claim is verified with $x=x^0$ and $S=M_{\xi,\nu}$.

Since $x^0$ was arbitrary, the result follows.
\end{proof}
Similarly to what we observed after Theorem \ref{cor1}, also here, if the map is continuous, one has a stronger result, as we can replace, in the previous statement, ``chain-recurrence relation" by ``recurrence relation" (see for instance, Lemma 2.1 in \cite{Hirsch}).

\vspace{15mm}
In the table below we summarize our main results (right) and their continuous versions (left).

\begin{table}[H]
\renewcommand{\arraystretch}{1.8} 
\begin{tabularx}{\textwidth}{|>{\Centering}X|>{\Centering}X|}
\hline
Assuming $f$ continuous & No assumptions on $f$ \\
\hline
There exists an almost periodic point & There exists a $\mathcal{C}$-recurrent point \newline (in fact also belonging to $GR(f)$) \\
Every point is in recurrence relation with a $\mathcal{C}$-recurrent point & Every point is in $\mathcal{C}$-recurrence relation with a $\mathcal{C}$-recurrent point \\
There exists a minimal subsystem & There exists a $\mathcal{C}$-transitive subsystem \\
Every point is in recurrence relation with a $\mathcal{C}$-transitive subsystem & Every point is in $\mathcal{C}$-recurrence relation with a $\mathcal{C}$-transitive subsystem \\
\hline
\end{tabularx}
\caption{Comparison of existence results in compact systems with continuity and with no assumptions on the map.}
\end{table}




\pagebreak

\end{document}